\documentclass[11pt,reqno]{amsart}

\newcommand\version{September 14, 2020}

\usepackage{amsmath,amsfonts,amsthm,amssymb,amsxtra}
\usepackage{bbm} 
\usepackage{tikz}

\newtheorem{theorem}{Theorem}

\newtheorem{lemma}[theorem]{Lemma}

\theoremstyle{definition}

\theoremstyle{remark}

\newtheorem{remark}[theorem]{Remark}

\renewcommand{\epsilon}{\varepsilon}

\renewcommand{\phi}{\varphi}
\newcommand{\R}{\mathbb{R}}

\begin{document}

\title[On the convolution inequality $f \geqslant f\star f$--- \version]{On the convolution inequality $\boldsymbol{f} \boldsymbol{\geqslant} \boldsymbol{f} \boldsymbol{\star} \boldsymbol{f} $}

\author{Eric A. Carlen}
\address[Eric A. Carlen]{Department of Mathematics, Hill Center,
Rutgers University, 110 Frelinghuysen Road, Piscataway, NJ 08854-8019, USA}
\email{carlen@math.rutgers.edu}

\author{Ian Jauslin}
\address[Ian Jauslin]{Department of Physics, Princeton University, Washington Road, Princeton, NJ 08544, USA}
\email{ijauslin@princeton.edu}

\author{Elliott H. Lieb}
\address[Elliott H. Lieb]{Departments of Mathematics and Physics, Princeton University, Washington Road, Princeton, NJ 08544, USA}
\email{lieb@princeton.edu}

\author{Michael P. Loss}
\address[Michael P. Loss]{School of Mathematics, Georgia Institute of Technology, Atlanta GA 30332}
\email{loss@math.gatech.edu}

\begin{abstract}
We consider the inequality $f \geqslant f\star f$  for real functions in $L^1(\R^d)$ where $f\star f$ denotes the convolution of $f$ with itself. We show that all such functions $f$ are non-negative, 
which is not the case for the same inequality in $L^p$ for any $1 < p \leqslant 2$, for which the convolution is defined. We also show that all solutions in $L^1(\R^d)$ satisfy $\int_{\R^d}f(x){\rm d}x \leqslant \tfrac12$.  Moreover, if 
$\int_{\R^d}f(x){\rm d}x = \tfrac12$, then $f$ must decay fairly slowly: $\int_{\R^d}|x| f(x){\rm d}x = \infty$, and this is sharp since for all 
$r< 1$, there are   solutions with  $\int_{\R^d}f(x){\rm d}x = \tfrac12$ and $\int_{\R^d}|x|^r f(x){\rm d}x <\infty$.  However, if 
$\int_{\R^d}f(x){\rm d}x  = : a < \tfrac12$, the decay at infinity can be much more rapid: we show that for all $a<\tfrac12$, there are solutions such that for some $\epsilon>0$, 
$\int_{\R^d}e^{\epsilon|x|}f(x){\rm d}x < \infty$. 
\end{abstract}

\thanks{\copyright\, 2020 by the authors. This paper may be
reproduced, in its entirety, for non-commercial purposes.\\
U.S.~National Science Foundation grants DMS-1764254 (E.A.C.),  DMS-1802170 (I.J.)  and  DMS-1856645 (M.P.L)  are gratefully acknowledged.}

\maketitle


\indent Our subject is  the set of real,  integrable solutions of the inequality 
\begin{equation}
  f(x)\geqslant f\star f(x)
  ,\quad\forall x\in\mathbb R^d\ ,
  \label{ineq}
\end{equation}
where $f\star f(x)$ denotes the convolution $f\star f(x) = \int_{\R^d} f(x-y) f(y){\rm d}y$. By Young's inequality \cite[Theorem 4.2]{LL96}, for all $1 \leqslant p\leqslant 2$ and all $f\in L^p(\R^d)$, $f \star f$ is well defined as an element of 
$L^{p/(2-p)}(\R^d)$.  Thus, one may consider the inequality \eqref{ineq} in $L^p(\R^d)$ for all $1 \leqslant p \leqslant 2$, but the case $p=1$ is special:  
the solution set of \eqref{ineq} is restricted in a number of surprising ways. Integrating both sides of \eqref{ineq}, one sees immediately that $\int_{\R^d} f(x){\rm d}x \leqslant 1$. 
We prove that, in fact, all integrable solutions satisfy $\int_{\R^d} f(x){\rm d}x \leqslant \tfrac12$, and this upper bound is sharp.

Perhaps even more surprising, we prove that all integrable solutions of \eqref{ineq} are non-negative.  This is {\em not true} for solutions in $L^p(\R^d)$, $ 1 < p \leqslant 2$. 
For  $f\in L^p(\R^d)$,  $1\leqslant p \leqslant 2$, the Fourier transform $\widehat{f}(k) = \int_{\R^d}e^{-i2\pi k\cdot x} f(x){\rm d}x$ is well defined as an element of $L^{p/(p-1)}(\R^d)$.  If $f$ solves the equation $f = f\star f$,
then $\widehat{f} = \widehat{f}^2$, and hence $\widehat{f}$ is the indicator function of a measurable set. By the Riemann-Lebesgue Theorem, if $f\in L^1(\R^d)$, then  
$\widehat{f}$ is continuous and vanishes at infinity, and the only such indicator function is the indicator function of the empty set. Hence the only integrable solution of 
$f = f\star f$ is the trivial solution $f= 0$. However, for $1 < p \leqslant 2$, solutions abound:  take $d=1$ and define $g$ to be the indicator function of the interval $[-a,a]$.  Define 
\begin{equation}\label{exam}
f(x) = \int_{\R} e^{-i2\pi k x} g(k){\rm d}k =  \frac{ \sin 2\pi xa}{\pi x}\ ,
 \end{equation}
which is not integrable, but which belongs to $L^p(\R)$ for all $p> 1$.  By the Fourier Inversion Theorem $\widehat{f} = g$. Taking products, one gets examples in any dimension.

To construct a family of solutions to \eqref{ineq}, fix $a,t> 0$,  and define  $g_{a,t}(k) = a e^{-2\pi |k| t}$. By \cite[Theorem 1.14]{SW71},
$$
f_{a,t}(x) =  \int_{\R^d} e^{-i2\pi k x} g_{a,t}(k){\rm d}k = a\Gamma((d+1)/2) \pi^{-(d+1)/2} \frac{t}{(t^2 + x^2)^{(d+1)/2}}\ .
$$
Since $g_{a,t}^2(k) =  g_{a^2,2t}$, $f_{a,t}\star f_{a,t} = f_{a^2,2t}$, Thus, $f_{a,t} \geqslant f_{a,t}\star f_{a,t}$ reduces to 
$$
\frac{t}{(t^2 + x^2)^{(d+1)/2}} \geqslant  \frac{2at}{(4t^2 + x^2)^{(d+1)/2}}
$$
which is satisfied for all $a \leqslant 1/2$. Since  $\int_{\R^d}f_{a,t}(x){\rm d}x =a$, this provides a class of solutions of \eqref{ineq} that are non-negative and satisfy 
\begin{equation}\label{half}
\int_{\R^d}f(x){\rm d}x \leqslant \frac12\ ,
\end{equation}
 all of which have fairly slow decay at infinity, so that in every case, 
\begin{equation}\label{tail}
 \int_{\R^d}|x|f(x){\rm d}x =\infty \ .
 \end{equation}

Our results show that this class of examples of integrable solutions of \eqref{ineq} is surprisingly typical  of {\em all} integrable solutions: every real integrable solution 
$f$  of \eqref{ineq} is positive, satisfies \eqref{half},
and if there is equality in \eqref{half}, $f$ also satisfies \eqref{tail}.  The positivity of all real solutions of \eqref{ineq} in $L^1(\R^d)$ may be considered surprising since it is 
false in $L^p(\R^d)$ for all $p > 1$, as the example \eqref{exam} shows.   We also show that when strict inequality holds in \eqref{half} for a solution $f$ of \eqref{ineq}, it is possible for 
$f$ to have rather fast decay; we construct examples such that $\int_{\R^d}e^{\epsilon|x|}f(x){\rm d}x < \infty$ for some $\epsilon> 0$.  The conjecture that integrable solutions of  \eqref{ineq}  
are necessarily positive was motivated by recent work \cite{CJL20,CJL20b} on a  partial differential equation involving a quadratic nonlinearity of $f\star f$ type, and the result proved here is the key to 
the proof of positivity for solutions of this partial differential equation; see \cite{CJL20}.  Autoconvolutions $f\star f$ have been studied extensively; see \cite{MV10}  and the work quoted there. However, the questions investigated by these authors are quite different from those  considered here.

\begin{theorem}\label{theo:positivity}
Let  $f$ be a real valued function in $L^1(\R^d)$ such that
\begin{equation}\label{uineq}
   f(x) - f\star f(x) =: u(x) \geqslant 0
\end{equation}
for all $x$.  Then  $\int_{\R^d} f(x)\ dx\leqslant\frac12$, 
and $f$ is given by the  series
\begin{equation}
f(x) = \frac{1}{2} \sum_{n=1}^\infty  c_n 4^n (\star^n u)(x)
\label{fun}
\end{equation} 
which converges in $L^1(\R^d)$, and
where the $c_n\geqslant0$ are the Taylor coefficients in the expansion of $\sqrt{1-x}$
\begin{equation}\label{3half}
  \sqrt{1-x}=1-\sum_{n=1}^\infty c_n x^n
  ,\quad
  c_n=\frac{(2n-3)!!}{2^nn!}  \sim n^{-3/2}
\end{equation}
In particular, $f$ is positive.  Moreover, if $u\geqslant 0$  is any integrable function with $\int_{\R^d}u(x){\rm d}x \leqslant \tfrac14$, then the sum on the right in \eqref{fun} defines an integrable function $f$ that satisfies \eqref{uineq}, and $\int_{\R^d}f(x){\rm d}x = \tfrac12$ if and only if $\int_{\R^d}u(x){\rm d}x= \tfrac14$.
\end{theorem}
\begin{proof} Note that $u$ is integrable. Let $a := \int_{\R^d}f(x){\rm d}x$ and $b := \int_{\R^d}u(x){\rm d}x \geqslant 0$. Fourier transforming,
\eqref{uineq} becomes
\begin{equation} \label{ft}
\widehat f(k) = \widehat f(k)^2 +\widehat u(k)\ .
\end{equation}
At $k=0$, $a^2 - a = -b$, so that $\left(a - \tfrac12\right)^2  = \tfrac14 - b$.  Thus $0 \leqslant b \leqslant\tfrac14$.  Furthermore, since  $u \geqslant 0$,
\begin{equation}
  |\widehat u(k)|\leqslant\widehat u(0) \leqslant \tfrac14
\end{equation}
and the first   inequality is strict  for $k\neq 0$. Hence for $k\neq 0$,  $\sqrt{1-4\widehat u(k)} \neq 0$.  By the Riemann-Lebesgue Theorem, $\widehat{f}(k)$ and $\widehat{u}(k)$ are both 
continuous and vanish at infinity, and hence we must have that 
\begin{equation}
  \widehat f(k)=\tfrac12-\tfrac12\sqrt{1-4\widehat u(k)}
  \label{hatf}
\end{equation}
for all sufficiently large $k$, and in any case $\widehat f(k)= \frac12\pm\frac12\sqrt{1-4\widehat u(k)}$.  But by continuity and the fact that  $\sqrt{1-4\widehat u(k)} \neq 0$ for any $k\neq 0$, the sign cannot switch.  
Hence \eqref{hatf} is valid for all $k$, including $k=0$, again by continuity.  At $k=0$,  $a = \tfrac12 - \sqrt{1-4b}$, which proves \eqref{half}.
The fact that $c_n$ as specified in \eqref{3half} satisfies $c_n \sim n^{-3/2}$ is a simple application of Stirling's formula, and it shows that the power series for $\sqrt{1-z}$ converges absolutely and 
uniformly everywhere on the closed unit disc. Since $|4 \widehat u(k)| \leqslant 1$,
${\displaystyle 
 \sqrt{1-4 \widehat u(k)} = 1 -\sum_{n=1}^\infty c_n (4 \widehat u(k))^n}$. Inverting the Fourier transform, yields \eqref{fun}, and since $\int_{\R^d} 4^n\star^n u(x){\rm d}x \leqslant 1$, 
 the convergence of the sum in $L^1(\R^d)$ follows from the convergence of $\sum_{n=1}^\infty c_n$. The final statement follows from the fact that if $f$ is defined in terms of $u$ in this manner, then \eqref{hatf} is 
 valid, and then \eqref{ft} and \eqref{uineq} are satisfied.
\end{proof}

\begin{theorem}\label{theo:decay}
Let $f\in L^1(\R^d)$ satisfy \eqref{ineq} and $\int_{\R^d} f(x)\ dx=\tfrac12$. Then 
$\int_{\R^d}|x| f(x)\ dx=\infty$.
\end{theorem}

\begin{proof}  If  $\int_{\R^d} f(x)\ dx=\tfrac12$, $\int_{\R^d} 4u(x)\ dx=1$, then $w(x) = 4u(x)$ is a probability density, and we can write $f(x) = \tfrac12\sum_{n=1}^\infty c_n \star^n w$.  Aiming for a contradiction, suppose that $|x|f(x)$ is integrable. Then $|x|w(x)$ is integrable.  Let $m:= \int_{\R^d}xw(x){\rm d} x$.  Since first moments add under convolution, the trivial inequality $|m||x| \geqslant m\cdot x$ yields 
$$|m|\int_{\R^d} |x| \star^nw(x){\rm d}x \geqslant  \int_{\R^d} m\cdot x \star^nw(x){\rm d}x = n|m|^2\ .$$
 It follows that  $\int_{\R^d} |x| f(x){\rm d}x \geqslant \frac{|m|}{2}\sum_{n=1}^\infty nc_n = \infty$. Hence $m=0$.

Suppose temporarily that in addition, $|x|^2w(x)$ is integrable. Let $\sigma^2$ be the variance of $w$; i.e., $\sigma^2 = \int_{\R^d}|x|^2w(x){\rm d}x$.
Define the function $\varphi(x) = \min\{1,|x|\}$.  Then 
$$
\int_{\R^d}|x| \star^n w(x){\rm d}x =  \int_{\R^d}|n^{1/2}x| \star^n w(n^{1/2}x)n^{d/2}{\rm d}x  \geqslant  n^{1/2} \int_{\R^d}\varphi(x)\star^n w(n^{1/2}x)n^{d/2}{\rm d}x.
$$
By the Central Limit Theorem,  since $\varphi$ is bounded and continuous,
\begin{equation}\label{CLT}
\lim_{n\to\infty}  \int_{\R^d}\varphi(x)\star^n w(n^{1/2}x)n^{d/2}{\rm d}x =  \int_{\R^d}\varphi(x) \gamma(x){\rm d}x =: C > 0
\end{equation}
where $\gamma(x)$ is a centered Gaussian probability density with variance $\sigma^2$. 

This shows that  there is a  $\delta> 0$  such that  for all sufficiently large $n$, $\int_{\R^d}|x| \star^n w(x){\rm d}x \geqslant \sqrt{n}\delta$, and then since $c_n\sim n^{-3/2}$, $\sum_{n=1}^\infty c_n \int_{\R^d}|x| \star^n w(x){\rm d}x= \infty$. 

 To remove the hypothesis that $w$ has finite variance, note that if $w$ is a probability density with zero mean and infinite variance, $\star^n w(n^{1/2}x)n^{d/2}$ is ``trying'' to converge to a ``Gaussian of infinite variance''. In particular, one would expect that for all $R>0$, 
\begin{equation}\label{CLT2}
\lim_{n\to\infty} \int_{|x| \leqslant R}\star^n w(n^{1/2}x)n^{d/2}{\rm d}x = 0\ , 
\end{equation} so that  the limit in \eqref{CLT} has the value $1$.  The proof then proceeds as above. The fact that \eqref{CLT2} is valid is a consequence  of Lemma~\ref{CLTL} below, which is closely based on  the proof of \cite[Corollary 1]{CGR08}.\end{proof}

\begin{theorem}\label{theo:decay3}
 Let $f\in L^1(\R^d)$ satisfy \eqref{uineq}, $\int_{\R^d} xu(x){\rm d}x = 0$  and  
$\int_{\R^d} |x|^2u(x){\rm d}x<\infty$. Then for all $0\leqslant p<1$,
  \begin{equation}
    \int_{\R^d} |x|^pf(x)\ dx<\infty.
  \end{equation}
\end{theorem}

\begin{proof}  We may suppose that $f$ is not identically $0$.  Let  $t := 4\int_{\R^d}u(x){\rm d}x \leqslant 1$.  Then $t> 0$.  Define $w := t^{-1}4u$; $w$ is a probability density and 
\begin{equation}\label{tfor}
f(x) = \tfrac12\sum_{n=1}^\infty c_n t^n \star^n w(x)\ . 
\end{equation}
By  hypothesis, $w$ has a  zero mean  and  variance  $\sigma^2 =   \int_{\R^d} |x|^2 w(x){\rm d}x < \infty$. Since variance is additive under convolution,
$$
\int_{\R^d} |x|^2 \star^n w(x){\rm d}x   = n\sigma^2\ .
$$
By H\"older's inequality,  for all $0 < p < 2$,
$\int_{\R^d} |x|^p \star^n w(x){\rm d}x   \leqslant  (n\sigma^2)^{p/2}$.
It follows that for $0 < p < 1$,
$$
\int_{\R^d} |x|^p f(x){\rm d}x   \leqslant  \frac12(\sigma^2)^{p/2} \sum_{n=1}^\infty n^{p/2} c_n  < \infty\ ,
$$
again using the fact that $c_n\sim n^{-3/2}$. 
\end{proof}

\begin{remark} In the subcritical case $\int_{\R^d}f(x){\rm d}x < \tfrac12$, the hypothesis that $\int_{\R^d} x u(x) {\rm d}x = 0$ is superfluous, and one can conclude more. In this case the quantity $t$  in \eqref{tfor} satisfies $0 < t < 1$, and  if we let $m$ denote the mean of $w$,
$\int_{\R^d} |x|^2 \star^n w(x){\rm d}x   =n^2|m|^2+ n\sigma^2$.  For $0<t<1$, $\sum_{n=1}^\infty n^2 c_n t^n < \infty$ and we conclude that $\int_{\R^d} |x|^2 f(x){\rm d}x < \infty$.  Finally, the final statement of Theorem~\ref{theo:positivity} shows that critical case functions $f$ satisfying the hypotheses of Theorem~\ref{theo:decay} are readily constructed.
\end{remark}

Theorem \ref{theo:decay} implies that when $\int f=\frac12$, $f$ cannot decay faster than $|x|^{-(d+1)}$.  However, integrable solutions $f$  of \eqref{ineq} such that $\int_{\R^d}f(x){\rm d}x < \tfrac12$
can decay more rapidly, as indicated in the previous remark. In fact, they may even have  finite exponential moments,  as we now show.

Consider a non-negative, integrable function $u$, which integrates to $r<\frac14$, and satisfies
\begin{equation}
  \int_{\R^d} u(x)e^{\lambda|x|}{\rm d}x < \infty
\end{equation}
for some $\lambda>0$.
The Laplace transform of $u$ is
$ \widetilde u(p):=\int e^{-px}u(x)\ {\rm d} x$ which is analytic for $|p|<\lambda$, and $\widetilde u(0) < \tfrac14$.
Therefore, there exists $0<\lambda_0\leqslant \lambda$ such that, for all $|p|\leqslant\lambda_0$,
  $\widetilde u(p)<\tfrac14$.
By Theorem \ref{theo:positivity},
${\displaystyle 
  f(x):=\frac12\sum_{n=1}^\infty 4^nc_n(\star^n u)(x)}$
is an integrable solution of \eqref{ineq}.  For  
 $|p|\leqslant\lambda_0$, it has a well-defined Laplace transform $ \widetilde f(p)$ given by 
\begin{equation}
  \widetilde f(p)=\int e^{-px}f(x)\ dx=\frac12(1-\sqrt{1-4\widetilde u(p)})
\end{equation}
which is analytic for $|p|\leqslant \lambda_0$.
Note that
${\displaystyle e^{s|x|} \leqslant \prod_{j=1}^d e^{|sx_j|} \leqslant \frac{1}{d}\sum_{j=1}^d e^{d|sx_j|} \leqslant \frac{2}{d}\sum_{j=1}^d \cosh(dsx_j)}$.
Thus, for $|s|< \delta := \lambda_0/d$, $\int_{\R^d}  \cosh(dsx_j)f(x){\rm d}x < \infty$ for each $j$, and hence
$|s| < \delta$, $\int_{\R^d} e^{s|x|}f(x){\rm d}x < \infty$. 

However, there are no integrable solutions of \eqref{ineq} that have compact support: We have seen that all  solutions of \eqref{ineq} are non-negative, and if $A$ is the support of a non-negative integrable function, the Minkowski sum $A+A$ is the support of $f\star f$.

\begin{remark} One might also consider the inequality $f \leqslant f \star f$ in $L^1(\R^d)$, but it is simple to construct  solutions that  have both signs. Consider any radial Gaussian probability density $g$,  
Then $g\star g(x) \geqslant g(x)$ for all sufficiently large $|x|$, and taking  $f:= ag$ for $a$ sufficiently large, we obtain $f< f\star f$ everywhere. Now on a small neighborhood of the origin, replace the value of 
$f$ by $-1$. If the region is taken small enough, the new function $f$ will still satisfy $f < f\star f$ everywhere.  
\end{remark}

We close with a lemma validating \eqref{CLT2} that is closely based on a  construction in \cite{CGR08}.  

\begin{lemma}\label{CLTL}
Let $w$ be a mean zero, infinite variance probability density on $\R^d$. Then for all $R>0$, \eqref{CLT2} is valid.
\end{lemma}

\begin{proof}   Let $X_1,\dots,X_n$ be $n$ independent samples from the density $w$, and let $B_R$ denote the centered ball of radius $R$. The quantity in \eqref{CLT2} is $p_{n,R} := \mathbb{P}(n^{-1/2}\sum_{j=1}^n X_j\in B_R)$.
Let $\widetilde X_1,\dots,\widetilde X_n$ be another $n$ independent samples from the density $w$, independent of the first $n$. Then also $p_{n,R} := \mathbb{P}(-n^{-1/2}\sum_{j=1}^n \widetilde X_j\in B_R)$. By the independence and the triangle inequality,
$$
p_{n,R}^2 \leq  \mathbb{P}(n^{-1/2}\sum_{j=1}^n (X_j -\widetilde X_j)\in B_{2R})\ .
$$
The random variable $X_1 - \widetilde X_1$ has zero  mean and infinite variance and an even density. Therefore, without loss of generality, we may assume that $w(x) = w(-x)$ for all $x$. 

Pick $\epsilon>0$, and choose a large value $\sigma_0$ such that $(2\pi \sigma_0^2)^{-d/2}R^d|B| < \epsilon/3$, where $|B|$ denotes the volume of the unit ball $B$.   The point of this is that if 
$G$ is a centered Gaussian random variable with variance {\em at least} $\sigma_0^2$, the probability that $G$ lies in {\em any} particular translate $B_R+y$ of the ball of radius $R$  is no more than $\epsilon/3$.  Let $A\subset \R^d$ be a centered cube such that 
$$
 \int_{A}|x|^2w(x){\rm d}x =: \sigma^2 \geqslant 2\sigma_0^2 \quad{\rm and}\quad \int_{A}w(x){\rm d}x > \tfrac34\ ,
 $$
 and note that since $A$ and $w$ are even,  ${\displaystyle  \int_{A} x w(x){\rm d}x = 0}$.

 It is then easy to find mutually independent random variables $X$, $Y$ and $\alpha$ such that 
$X$ takes values in $A$ and, has zero mean and variance $\sigma^2$, $\alpha$ is a Bernoulli variable with success probability $\int_{A}w(x){\rm d}x$, and finally such that $\alpha X + (1-\alpha)Y $  has the probability density $w$. Taking independent identically distributed (i.i.d.) sequences of such random variables, $w(n^{1/2}x)n^{d/2}$ is the probability density of
${\displaystyle W_n :=  n^{-1/2}\sum_{j=1}^n \alpha_j X_j +  n^{-1/2} \sum_{j=1}^n(1-\alpha_j)Y_j}$,  and we seek to estimate
 the expectation of $1_{B_R}(W_n)$. We first take the conditional expectation, given the values of the $\alpha$'s and the $Y$'s, and we define $\hat{n} =  \sum_{j=1}^n\alpha_j$. These conditional expectations have the form 
${\mathbb E}\left[ 1_{B_R + y}\left(\sum_{j=1}^n n^{-1/2}\alpha_j X_j \right)\right]$
for some translate $B_R +y$ of  $B_R$, the ball of radius $R$.   The sum $n^{-1/2}\sum_{j=1}^n \alpha_j X_j$ is actually  the sum of $\hat{n}$ i.i.d. random variables with mean zero and variance  $\sigma^2/n$.  The probability that $\hat{n}$ is significantly less than $\frac34 n$ is negligible for large $n$; by classical estimates associated with the Law of Large Numbers,   for all $n$ large enough, the probability that $\hat{n} < n/2$  is no more than $\epsilon/3$.  Now let $Z$ be a Gaussian random variable with mean zero and variance 
$\sigma^2\hat{n}/n$ which is at least $\sigma^2_0$ when $\hat{n} \geqslant n/2$. Then by the multivariate version \cite{R19} of the  Berry-Esseen Theorem \cite{B41,E42}, a version of the Central Limit Theorem 
with rate information,  there is a constant $K_d$ depending only on $d$ such that
$${\textstyle 
\left|{\mathbb E}\left[ 1_{B_R + y}\left(\sum_{j=1}^n n^{-1/2}\alpha_j X_j \right)\right]  -  {\mathbb P}\{Z \in B_R + y\}\right|  \leqslant K_d \hat{n} \frac{{\mathbb E}|X_1|^3}{n^{3/2}}  \leqslant K_d  \frac{{\mathbb E}|X_1|^3}{n^{1/2}}
\ .}
$$
Since $A$ is bounded, ${\mathbb E}|X_1|^3 < \infty$,  and hence for all sufficiently large $n$,  when  $\hat{n} \geqslant n/2$.
$${\textstyle 
{\mathbb E}\left[ 1_{B_R + y}\left(\sum_{j=1}^n n^{-1/2}\alpha_j X_j \right)\right]  \leqslant \frac23 \epsilon
\ .}
$$
Since this is uniform in $y$, we finally obtain  ${\mathbb P}(W_n \in B_R) \leq \epsilon$ for all sufficiently large $n$. Since $\epsilon>0$ is arbitrary, \eqref{CLT2} is proved.
\end{proof}

We close by thanking the anonymous referee for useful suggestions.


\bibliographystyle{amsalpha}

\end{document}